\documentclass{amsart}
  
  \usepackage{amsthm} 
  \usepackage{amsmath} 
  \usepackage{amssymb} 
  \usepackage{mathtools} 
   \usepackage{amsfonts}
  \usepackage{color}
   
  \usepackage[cal=boondox]{mathalfa}

 \usepackage{graphicx}
  \usepackage{xspace}
  \usepackage[all]{xy}
  \usepackage{pinlabel}
  \usepackage{enumerate}

\usepackage[margin=1.3in]{geometry}
  \usepackage{hyperref}

  \newcommand{\calA}{\mathcal{A}}

  \newcommand{\calG}{\mathcal{G}}

  \newcommand{\calL}{\mathcal{L}}

  \newcommand{\calS}{\mathcal{S}}

  \newcommand{\RR}{\mathbb{R}}

  \newcommand{\ZZ}{\mathbb{Z}}

  \newtheorem{theorem}{Theorem}[section]
  \newtheorem{proposition}[theorem]{Proposition}
  \newtheorem{corollary}[theorem]{Corollary}
  \newtheorem{lemma}[theorem]{Lemma}

  \theoremstyle{definition}
  \newtheorem{definition}[theorem]{Definition}

  \newtheorem*{claim*}{Claim}
  
  \newtheorem{example}[theorem]{Example}
  
  \newtheorem*{question*}{Question}
  \newtheorem*{answer*}{Answer}
  \newtheorem*{application*}{Application}

  \theoremstyle{remark}
  
  \newtheorem*{remark*}{Remark}
  
  \newcommand{\secref}[1]{Section~\ref{#1}}
  \newcommand{\thmref}[1]{Theorem~\ref{#1}}
  \newcommand{\corref}[1]{Corollary~\ref{#1}}
  \newcommand{\lemref}[1]{Lemma~\ref{#1}}
  \newcommand{\propref}[1]{Proposition~\ref{#1}}

  \newcommand{\eqnref}[1]{Equation~\eqref{#1}}

  \DeclareMathOperator{\id}{id}

  \DeclareMathOperator{\Proj}{Proj}  
  \DeclareMathOperator{\diam}{diam}    
  \DeclareMathOperator{\Ball}{Ball}  

  \newcommand{\Map}{\ensuremath{\operatorname{Map}}\xspace}    
   
  \newcommand{\Teich}{Teich\-m\"u\-ller\ } 
  
  \newcommand{\param}{{\mathchoice{\mkern1mu\mbox{\raise2.2pt\hbox{$
  \centerdot$}}
  \mkern1mu}{\mkern1mu\mbox{\raise2.2pt\hbox{$\centerdot$}}\mkern1mu}{
  \mkern1.5mu\centerdot\mkern1.5mu}{\mkern1.5mu\centerdot\mkern1.5mu}}}

  \newcommand{\from}{\colon\thinspace}

  \DeclarePairedDelimiter\abs{\lvert}{\rvert}
  \DeclarePairedDelimiter\norm{\lVert}{\rVert}  
  \DeclarePairedDelimiter\Norm{\Big\lVert}{\Big\rVert}

  \newcommand{\calg}{{\mathcal{g}}}  
  \newcommand{\aphi}{{\mathcal{a}_\phi}}
  \newcommand{\Aphi}{{\mathcal{A}_\phi}}
  \newcommand{\Gphi}{{\mathcal{G}_\phi}}
  \newcommand{\Sn}{{\mathcal{S}_n}}
  \newcommand{\CS}{{\mathcal{C}(S)}}
  \newcommand{\PG}{\Proj_\Gphi\!}

\begin{document}

\title{Geodesics in the mapping class group}


 \author   {Kasra Rafi}
 \address{Department of Mathematics, University of Toronto, Toronto, ON }
 \email{rafi@math.toronto.edu}

 \author   {Yvon Verberne}
 \address{Department of Mathematics, University of Toronto, Toronto, ON }
 \email{yvon.verberne@mail.utoronto.ca}
 
 
  \date{\today}

\begin{abstract}
We construct explicit examples of geodesics in the mapping class group and show that 
the shadow of a geodesic in mapping class group to the curve graph does not have to 
be a quasi-geodesic. We also show that the quasi-axis of a pseudo-Anosov 
element of the mapping class group may not have the strong contractibility property. 
Specifically, we show that, after choosing a generating set carefully, one can find a 
pseudo-Anosov homeomorphism $\phi$, a sequence of points $w_k$ and a sequence 
of radii $r_k$ so that the ball $B(w_k, r_k)$ is disjoint from a quasi-axis 
$\aphi$ of $\phi$, but for any projection map from mapping class group to 
$\aphi$, the diameter of the image of $B(w_k, r_k)$ grows like $\log(r_k)$.
\end{abstract}

\maketitle

\section{Introduction}
Let $S$ be a surface of finite type and let $\Map(S)$ denote the (pure) mapping class 
group of $S$, that is, the group of orientation preserving self homeomorphisms of $S$ 
fixing the punctures of $S$, up to isotopy.
This is a finitely generated group \cite{D1} and, after choosing a generating set, 
the word length turns $\Map(S)$ into a metric space. The geometry of $\Map(S)$
has been a subject of extensive study.  Most importantly, in \cite{MM2}, Masur
and Minsky gave an estimate for the word length of a mapping class using the
sub-surface projection distances and constructed efficient quasi-geodesics in the 
mapping class group, called hierarchy paths, connecting the identity to any
given mapping class. The starting point of the construction of a hierarchy path is 
a geodesic in the curve graph of $S$ which is known to be a Gromov hyperbolic space
\cite{MM1}. Hence, by construction, the shadow of a hierarchy path to the curve graph 
is nearly a geodesic. 

It may seem intuitive that any geodesic in the mapping class group should also 
have this property, considering that similar statements have been shown to be true 
in other settings. For example, it is known that the shadow of a geodesic
in \Teich space with respect to the \Teich metric 
is a re-parametrized quasi-geodesic in the curve graph
\cite{MM1}. The same is true for any geodesic in \Teich space
with respect to the Thurston metric \cite{LRT}, for any line of minima in \Teich 
space \cite{CRS}, for a grafting ray \cite{CDR}, or for the set of short curves in a hyperbolic 
$3$--manifold homeomorphic to $S \times \RR$ \cite{M3}. 
However, it is difficult to construct explicit examples of geodesics in $\Map(S)$ and 
so far, all estimates for the word length of an element have been up to a multiplicative error. 

In this paper, we argue that one should not expect geodesics in $\Map(S)$ to be 
well-behaved in general. Changing the generating set changes the metric on $\Map(S)$ 
significantly and a geodesic with respect to one generating set is only a quasi-geodesic 
with respect to another generating set. Since $\Map(S)$ is not Gromov hyperbolic 
(it contains flats), its quasi-geodesics are not well behaved in general. Similarly, 
one should not expect that the geodesics with respect to an arbitrary generating set 
to behave well either. 

\begin{figure}[ht]
\setlength{\unitlength}{0.01\linewidth}
\begin{picture}(100, 17)
\put(37,-2){\includegraphics[width=25\unitlength]{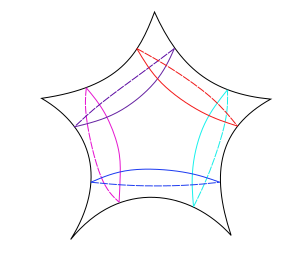}}

\put(42,13){$\alpha_1$}
\put(51,.5){$\alpha_2$}

\put(51.5, 15){$\alpha_3$} 
\put(41, 4){$\alpha_4$}
\put(57, 8){$\alpha_5$}

\end{picture} 
\caption{The curves $\alpha_1, \ldots, \alpha_5$ used to generate $\Sn$.} 
\label{Fig:Surface} 
\end{figure}
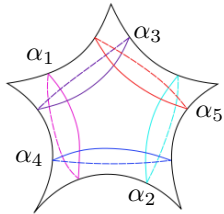

We make this explicit in the case where $S=S_{0,5}$ is the five-times punctured
sphere. Consider the curves $\alpha_1, \dots, \alpha_5$ depicted in 
Figure~\ref{Fig:Surface}. Fix an integer $n \gg 1$ (to be determined in the proof of 
Theorem \ref{Thm:Main-Theorem}), and consider the following generating set for 
$\Map(S)$
\begin{equation*}
\Sn= \Big\{ D_{\alpha_i}, s_{i,j} : i,j \in \mathbb{Z}_5, |i-j| = 1 \mod 5 \Big\}
\end{equation*}
where $s_{i,j}=D_{\alpha_i}^{n} D_{\alpha_j}^{-1}$,
and $D_\alpha$ is a Dehn twist around a curve $\alpha$. Since we are considering the 
pure mapping class group, the set $\left\{ D_{\alpha_i} \right\}_{i=1}^5$ already generates $\Map(S)$. 
We denote the distance on $\Map(S)$ induced by the generating set $\Sn$ by
$d_\Sn$. By an $\Sn$--geodesic, we mean a geodesic with respect to
this metric. 

\begin{theorem}\label{thm:shadow}
There is an $n \gg 1$ so that, for every $K, C >0$, there exists an $\Sn$--geodesic 
\[
\calG \from [0,m] \to \Map(S)
\] 
so that the shadow of $\calG$ to the curve 
graph $\CS$ is not a re-parametrized $(K,C)$--quasi-geodesic. 
\end{theorem}

Even though the mapping class group is not Gromov hyperbolic, it does have
hyperbolic directions. There are different ways to make this precise. 
For example, Behrstock proved \cite{Be} that in the direction of every pseudo-Anosov, 
the divergence function in $\Map(S)$ is super-linear. Another way to make this notion precise is to examine whether geodesics in $\Map(S)$ have the \emph{contracting property}.

This notion is defined analogously with Gromov hyperbolic spaces where, for every 
geodesic $\calG$ and any ball disjoint from $\calG$, the closest point projection of the ball 
to $\calG$ has a uniformly bounded diameter. However, often it is useful to work 
with different projection map. We call a map 
\[
\Proj \from X \to \calG
\]
from a metric space $X$ to any subset $\calG \subset X$ a 
\emph{$(d_1, d_2)$--projection map}, $d_1, d_2>0$, if for every $x \in X$ and $g \in \calG$, 
we have 
\[
d_X\big(\Proj (x), g \big) \leq d_1 \cdot d_X(x,g) + d_2. 
\]
This is a weak notion of projection since $\Proj$ is not even assumed to be coarsely 
Lipschitz. By the triangle inequality, the closest point projection is always a 
$(2,0)$--projection. 

\begin{definition}
A subset $\calG$ of a metric space 
$X$ is said to have the \emph{contracting property} 
if there is a constant $\rho <1$, a constant $B>0$ and a projection map 
$\Proj\from X \to \calG$ such that, for any ball $B(x,R)$ of radius $R$
 disjoint from $\calG$, the projection of a ball of radius $\rho R$, $B(x, \rho R)$, 
has a diameter at most $B$,
\[
 \diam_\Sn \big(\Proj(B(x,\rho R))\big) \leq B. 
\]
We say $\calG$ has the \emph{strong contracting property} if $\rho$ can be taken to 
be $1$. 
\end{definition}

The axis of a pseudo-Anosov element has the contracting property in many settings. 
This has been shown to be true in the setting of \Teich space by 
Minsky \cite{M1}, in the setting of the pants complex by Brock, Masur, and Minsky \cite{BM}
and in the setting of the mapping class group by Duchin and Rafi \cite{DR1}.

In work by Arzhantseva, Cashen and Tao, they asked if the axis of a pseudo-Anosov 
element in the mapping class group has the strong contracting property and showed
that a positive answer would imply that the mapping class group is growth tight \cite{ACT}.
Additionally, using the work of Yang \cite{Yang}, one can show that if one pseudo-Anosov
element has a strongly contracting axis with respect to some generating set, then a generic 
element in mapping class group has a strongly contracting axis with respect to this 
generating set. Similar arguments would also show that the mapping class group with 
respect to this generating set has purely exponential growth. 

However, using our specific generating set, we show that this does not always hold:

\begin{theorem}\label{Thm:Main-Theorem}
For every $d_1, d_2>0$, there exists an $n \gg 1$, a pseudo-Anosov map $\phi$, 
a constant $c_n>0$, a sequence  of elements $w_k \in \Map(S)$ and a sequence of 
radii $r_k>0$ where $r_k \to \infty$ as $k \to \infty$ such that the following holds. 
Let $\aphi$ be a quasi-axis for $\phi$ in $\Map(S)$ and let 
$\Proj_\aphi\from \Map(S) \to \aphi$ be any $(d_1, d_2)$--projection map. 
Then the ball of radius $r_k$ centered at $w_k$, $B(w_k,r_k)$, is disjoint
from $\aphi$ and 
\begin{equation*}
\diam_{\calS_n} \Big(\Proj_\aphi \big(B(w_k, r_k)\big)\Big) \geq c_n \log(r_k).
\end{equation*}
\end{theorem}

We remark that, since $\aphi$ has the contracting property \cite{DR1}, 
the diameter of the projection can grow at most logarithmically with respect to the radius $r_k$
(see \corref{Cor:Diam-Proj}), hence the lower-bound achieved by the above
theorem is sharp. 

\subsection*{Outline of proof} 
To find an exact value for the word-length of an element $f \in \Map(S)$, 
we construct a homomorphism 
\[
h\from \Map(S) \to \mathbb{Z},
\]  
where a large value for $h(f)$ guarantees a large value for the word length of $f$. 
At times, this lower bound is realized and an explicit geodesic in $\Map(S)$ is 
constructed (see \secref{Sec:Homomorphism}). 
The pseudo-Anosov element $\phi$ is defined as 
\[
\phi = D_{\alpha_5}D_{\alpha_4}D_{\alpha_3}D_{\alpha_2}D_{\alpha_1}.
\]
In \secref{Sec:pA} we find an explicit invariant train-track for $\phi$ to show that
$\phi$ is a pseudo-Anosov. In \secref{Sec:Shadows}, we use the geodesics constructed in 
\secref{Sec:Homomorphism} to show that the shadows of geodesics in $\Map(S)$ are not
necessarily quasi-geodesics in the curve complex. In \secref{Sec:pA_axis}, we begin 
by showing that $\phi$ acts loxodromically on $\Map(S)$, that is, it has a quasi-axis 
$\aphi$ which fellow travels the path $\{ \phi^i \}$. We finish \secref{Sec:pA_axis} by 
showing that the bound in our main theorem is sharp. In \secref{Sec:Upper_bound}, 
we set up and complete the proof of \ref{Thm:Main-Theorem}.

\subsection*{Acknowledgements} We thank Sam Taylor for helpful conversations and 
Camille Horbez and Chris Leininger for their comments on an earlier version of this 
paper. Kasra Rafi was partially supported by NSERC Discovery grant RGPIN 06486.

\section{Finding Explicit Geodesics} \label{Sec:Homomorphism}

In this section, we develop the tools needed to show that certain paths in $\Map(S)$
are geodesics. We emphasize again that, in our paper, $S$ is the five-times punctured 
sphere and $\Map(S)$ is the pure mapping class group. That is, all 
homeomorphisms are required to fix the punctures 
point-wise. 

By a \emph{curve} on $S$ we mean a free homotopy class of a non-trivial, 
non-peripheral simple closed curve. 
Fix a labelling of the 5 punctures of $S$ with elements of 
$\ZZ_5$, the cyclic group of order 5. Any curve $\gamma$ on $S$ cuts the surface 
into two surfaces; one copy of $S_{0,3}$ containing two of the punctures from $S$ and one copy 
of $S_{0,4}$ which contains three of the punctures from $S$. 

\begin{definition} \label{Def:Type}
We say that a curve $\gamma$ on $S$ is an \emph{$(i,j)$--curve}, $i,j \in \ZZ_5$, if the 
component of $(S - \gamma)$ that is a three-times punctured sphere contains the punctures 
labeled $i$ and $j$. Furthermore, if $|i-j| \equiv 1 \mod 5$ we say that $\gamma$ 
\emph{separates two consecutive punctures}, and if $|i-j| \equiv 2\mod 5$ we say that 
\emph{$\gamma$ separates two non-consecutive punctures}. 
\end{definition}

In \cite{L1}, Luo gave a simple presentation of the mapping class group where
the generators are the set of all Dehn twists 
\[
\calS = \{ D_{\gamma} : \gamma \text{ is a curve} \}
\]
and the relations are of a few simple types. In our setting, we only have the following 
two relations:
\begin{itemize}
\item (Conjugating relation) For any two curves $\beta$ and $\gamma$, 
\[
D_{D_{\gamma}(\beta)} = D_{\gamma} D_{\beta} D_{\gamma}^{-1}.
\]
\item (The lantern relation) Let $i, j, k, l, m$ be distinct elements in $\ZZ_5$ and $\gamma_{i,j}$, $\gamma_{j,k}$
$\gamma_{k,i}$ and $\gamma_{l,m}$ be curves of the type indicated by the indices. 
Further assume that each pair of curves among 
$\gamma_{i,j}$, $\gamma_{j,k}$ and $\gamma_{k,i}$ intersect twice and that they are
all disjoint from $\gamma_{l,m}$. Then 
\[
D_{\gamma_{i,j}}D_{\gamma_{j,k}}D_{\gamma_{k,i}} = D_{\gamma_{l,m}}.
\]
\end{itemize}

Using this presentation, we construct a homomorphism from $\Map(S)$ into 
$\ZZ$.

\begin{theorem}\label{thm:homo}
There exists a homomorphism $h \from \Map(S) \to \ZZ$ whose
restriction to the generating set $\calS$ is as follows:
\begin{align*}
D_{\gamma} &\longmapsto 1 &&\text{ if } \gamma \text{ separates two consecutive punctures}\\
D_{\gamma} &\longmapsto -1 &&\text{ if } \gamma \text{ separates two non-consecutive punctures}
\end{align*}
\end{theorem}

\begin{proof}
To show that $h$ extends to a homomorphism, it suffices to show that $h$ preserves 
the relations stated above.

First, we check the conjugating relation. Let $\beta$ and $\gamma$ be a pair of curves.
Since, $D_\gamma$ is a homeomorphism fixing the punctures, 
if $\beta$ is an $(i,j)$--curve, so is $D_{\gamma}(\beta)$. In particular, 
$h(D_{D_{\gamma}(\beta)}) = h(\beta)$.  Hence, 
\begin{align*}
h\big(D_{D_{\gamma}(\beta)}\big) = h(D_{\beta})
&= h(D_{\gamma}) + h(D_{\beta}) - h(D_{\gamma})\\ 
&= h(D_{\gamma}) + h(D_{\beta}) + h\big(D_{\gamma}^{-1}\big)
= h\big(D_{\gamma} D_{\beta} D_{\gamma}^{-1}\big). 
\end{align*}

We now show that $h$ preserves the lantern relation. For any three punctures of 
$S$ labeled $i , j, k \in \ZZ_5$, two of these punctures are consecutive. 
Without loss of generality, suppose $|i-j| = 1 \mod 5$. There are two cases:
\begin{enumerate}
\item Assume $k$ is consecutive to one of $i$ or $j$. That is, without loss of generality, 
suppose $|j-k| = 1 \mod 5$. Then $|i-k| = 2 \mod 5$ and the remaining two punctures, 
$l$ and $m$, are consecutive: $|l-m| = 1 \mod 5$. Thus
\begin{align*}
h(D_{\gamma_{i,j}}D_{\gamma_{j,k}}D_{\gamma_{k,i}}) &= h(D_{\gamma_{i,j}})+h(D_{\gamma_{j,k}})+h(D_{\gamma_{k,i}})\\
&=1 + 1 + (-1)\\
&=1 = h(D_{\gamma_{l,m}}).
\end{align*}

\item Otherwise, $|j-k| = 2 \mod 5$ and $|i-k| = 2 \mod 5$, so that the remaining two punctures, $l$ and $m$, are nonconsecutive: $|l-m| = 2 \mod 5$. Thus
\begin{align*}
h(D_{\gamma_{i,j}}D_{\gamma_{j,k}}D_{\gamma_{k,i}}) &= h(D_{\gamma_{i,j}})+h(D_{\gamma_{j,k}})+h(D_{\gamma_{k,i}})\\
&=1 + (-1) + (-1)\\
&=(-1)= h(D_{\gamma_{l,m}}).
\end{align*}

\end{enumerate}
Thus, $h$ preserves the lantern relation.
\end{proof}

Now, we switch back to the generating set $\Sn$ given in the introduction.
The homomorphism of \thmref{thm:homo} gives a 
lower bound on the word length of elements in $\Map(S)$.
Note that 
\[
h(s_{i,j}) = (n-1) \qquad\text{and}\qquad h(D_{\alpha_i}) = 1.
\] 

\begin{lemma}\label{cor1}
For any $f \in \Map(S)$, let 
\[
h(f)=q(n-1)+r
\]
for integer numbers $q$ and $r$ where
 $0 \leq \abs{r} < \frac{n-1}{2}$. Then $\norm{f}_\Sn \geq |q|+|r|$.
\end{lemma}

\begin{proof}
First we show that, if $h(f) = a(n-1)+b$ for integers $a$ and $b$, then 
$|a|+|b| \geq |q| + |r|$. To see this, consider such a pair $a$ and $b$ where
$|a|+|b|$ is minimized. If $a<q$, then 
$|b| \geq |(n-1) +r| > \frac{n-1}{2}$. Therefore, we can increase $a$ by $1$ and decrease 
$b$ by $n-1$ to decrease the quantity $|a|+|b|$, which is a contradiction. 
Similarly, if $a>q$, then $|b| > (n-1)/2$ and we can decrease $a$ by $1$ and 
increase $b$ by $n-1$ to decrease the quantity $|a|+|b|$, which again is a contradiction. 
Hence, $a=q$ and subsequently $b=r$. 

Now, write $f = g_1 g_2 \ldots g_k$, where $g_i \in \Sn$ or $g_i^{-1} \in \Sn$
and $k= \norm{f}_\Sn$. For each $g_i$, $h(g_i)$ takes either value 
$1$, $(-1)$, $(n-1)$ or $(1-n)$. Hence, there are integers $a'$ and $b'$
so that
\[
h(f) = h(g_1) + h(g_2) + \ldots + h(g_k) = a'(n-1) +b',
\]
where $k \geq |a'|+|b'|$. But, as we saw before, we also have $|a'|+|b'| \geq |q| + |r|$. 
Hence $k \geq |q| + |r|$. 
\end{proof}

This lemma allows us to find explicit geodesics in $\Map(S)$. We demonstrate this
with an example. 

\begin{example}\label{example1}
Let $f = D_{\alpha_1}^{n^k-1} \in \Map(S)$. We have, 
\[
h(f) = n^k-1 = (n-1) (n^{k-1} + n^{k-2} + \dots + n^2+n+1).
\]
Therefore, by Lemma \ref{cor1} $\norm{f}_\Sn \geq n^{k-1} + n^{k-2} + \dots + n^2+n+1$. 
On the other hand, (assuming $k$ is even to simplify notation), we have
\begin{align*}
D_{\alpha_1}^{n^k-1} 
& = \left(D_{\alpha_1}^{n^k} D_{\alpha_2}^{-n^{(k-1)}}\right) \!
      \left(D_{\alpha_2}^{n^{(k-1)}} D_{\alpha_1}^{-n^{(k-2)}}\right) \dots
      \left(D_{\alpha_1}^{n^2} D_{\alpha_2}^{-n}\right)  \!
      \Big(D_{\alpha_2}^{n} D_{\alpha_1}^{-1}\Big) \\
&= s_{1,2}^{n^{(k-1)}} \, s_{2,1}^{n^{(k-2)}} \dots s_{1,2}^{n} \, s_{2,1}.
\end{align*}
Since we used exactly $(n^{k-1} + n^{k-2} + \ldots +n+1)$ elements in $\Sn$, 
we have found a geodesic path. However, notice there is a second geodesic path from 
the identity to $f$ (which works for every $k$), namely:
\begin{align*}
D_{\alpha_1}^{n^k-1} 
&= \left(D_{\alpha_1}^{n^k} D_{\alpha_2}^{-n^{(k-1)}}\right) \!
      \left(D_{\alpha_2}^{n^{(k-1)}} D_{\alpha_3}^{-n^{(k-2)}}\right) \! \dots 
      \left(D_{\alpha_{(k-1)}}^{n^2} D_{\alpha_k}^{-n}\right) \!
      \Big(D_{\alpha_k}^{n} D_{\alpha_{(k+1)}}^{-1}\Big)\\
&= s_{1,2}^{n^{(k-1)}} \, s_{2,3}^{n^{(k-2)}} \dots s_{(k-1),k}^{n} \, s_{k,k+1}.
\end{align*}
This shows that geodesics are not unique in $\Map(S)$. Either way, we have
established that 
\begin{equation} \label{Eq:Twist-Norm} 
\norm{D_{\alpha_1}^{n^k-1} }_\Sn = n^{(k-1)} + n^{(k-2)} + \dots + n + 1. 
\end{equation}
\end{example}

We now use a similar method to compute certain word lengths that will be useful later 
in the paper. Define 
\[
\phi = D_{\alpha_5} D_{\alpha_4} D_{\alpha_3} D_{\alpha_2}D_{\alpha_1}.
\]
We will show in the next section that $\phi$ is a pseudo-Anosov element of $\Map(S)$. 
We also use the notation 
\[
\phi^{k/5} = D_{\alpha_k} D_{\alpha_{k-1}} \dots D_{\alpha_1}
\]
where indices are considered to be in $\ZZ_5$. This is accurate when 
$k$ is divisible by $5$ but we use it for any integer $k$. 
For a positive integer $k$, define 
\[
m_k = n^{k} + n^{k-1} + \ldots + n + 1
\qquad\text{and}\qquad
\ell_k = n^k - n^{k-1} - n^{k-2} - \ldots - n -1,
\]
and let $w_k = D_{\alpha_1}^{m_k}$ and $u_k = D_{\alpha_1}^{\ell_k}$. Additionally,
we will define
\[
v_k = D_{\alpha_1}^{- \frac{k+1}2} D_{\alpha_2}^{- \frac{k+1}2}.
\] 
We will show that $u_k$ and $w_k$ are closer to a large power of $\phi$ than 
the identity even though they are both just a power of a Dehn twist. 

\begin{proposition}\label{thm:geodesics}
For $u_k$ and $w_k$ as above, we have
\[
\Norm{w_k \, \phi^{-(k+1)/5}}_\Sn=
\left \Vert w_k \, v_k \right\Vert_\Sn=
n^{k-1}+2n^{k-2}+ \ldots + (k-1)n + k,
\] 
and
\[
\Norm{\phi^{k/5}\, u_k }_\Sn=n^{k-1} - n^{k-3} - 2n^{k-4} - \ldots - (k-3)n - (k-2) + 1.
\] 
\end{proposition}

\begin{proof}
Note that
\begin{align*}
h\Big(w_k \phi^{-(k+1)/5}\Big)&= (n^{k}+n^{k-1}+ \ldots + n+ 1) - (k+1)\\
 &= (n-1)(n^{k-1}+2n^{k-2}+ \ldots + (k-1)n +k).
\end{align*}
Lemma \ref{cor1} implies that
\[
\Norm{w_k \phi^{-(k+1)/5}}_\Sn \geq n^{k-1}+2n^{k-2}+ \ldots + (k-1)n + k.
\]
On the other hand, since $m_k -1 = n \, m_{k-1}$, we have
\begin{align*}
w_k \, \phi^{-(k+1)/5}  &= D_{\alpha_1}^{m_k} 
       \big( D_{\alpha_1}^{-1} D_{\alpha_2}^{-1} \dots D_{\alpha_{k+1}}^{-1}\big)\\
  &= D_{\alpha_1}^{(m_k-1)}  
       \big(D_{\alpha_2}^{-1} D_{\alpha_3}^{-1} \dots D_{\alpha_{k+1}}^{-1}\big)\\  
  &= s_{1,2}^{m_{k-1}} D_{\alpha_2}^{m_{k-1}} 
       \big( D_{\alpha_2}^{-1} D_{\alpha_2}^{-1} \dots D_{\alpha_{k+1}}^{-1}\big)\\
  &= s_{1,2}^{m_{k-1}} D_{\alpha_1}^{(m_{k-1}-1)}  
       \big(D_{\alpha_3}^{-1} D_{\alpha_4}^{-1} \dots D_{\alpha_{k+1}}^{-1}\big)\\  
  &= s_{1,2}^{m_{k-1}} \, s_{2,3}^{m_{k-2}} \, D_{\alpha_3}^{m_{k-2}} 
       \big( D_{\alpha_3}^{-1} D_{\alpha_4}^{-1} \dots D_{\alpha_{k+1}}^{-1}\big)\\
  & \qquad \ldots \\
  &= s_{1,2}^{m_{k-1}} \, s_{2,3}^{m_{k-2}}   \dots s_{k-1,k}^{m_1}\,  s_{k, k+1}. 
\end{align*} 
Therefore, 
\begin{equation*}
\norm{w_k \, \phi^{-(k+1)/5} }_\Sn = m_{k-1} + \dots + m_1 + 1 = 
    n^{k-1}+2n^{k-2}+ \ldots (k-1)n + k.
\end{equation*}
To show that
\[
\left \Vert w_k \, v_k \right\Vert_\Sn=
n^{k-1}+2n^{k-2}+ \ldots (k-1)n + k
\]
is as above, but in place of applying $s_{i,i+1}$ for $1 \leq i \leq k$, we alternate between applying $s_{1,2}$ and $s_{2,1}$ to find
\[
w_k \, v_k  = s_{1,2}^{m_{k-1}} \, s_{2,1}^{m_{k-2}}   \dots s_{1,2}^{m_1}\,  s_{2, 1},
\]
which proves our claim. Similarly, we have
\begin{align*}
h\Big(\phi^{k/5} u_k\Big) &= k+ \big(n^k - n^{k-1} + \ldots - n - 1\big) \\
&= (n-1)\big(n^{k-1} - n^{k-3} - 2n^{k-4} - \ldots - (k-3)n - (k-2)\big) +1,
\end{align*}
and Lemma \ref{cor1} implies
\begin{equation*}
\norm{\phi^{k/5} u_k}_\Sn \geq n^{k-1} - n^{k-3} - 2n^{k-4} - \ldots - (k-3)n - (k-2) + 1.
\end{equation*}

On the other hand, since $\ell_k +1 = n \ell_{k-1}$, we have
\begin{align*}
\phi^{k/5} u_k  
  &= \big( D_{\alpha_k} \dots D_{\alpha_2} D_{\alpha_1} \big) D_{\alpha_1}^{\ell_k} \\
  &= \big(D_{\alpha_k}\dots D_{\alpha_3} D_{\alpha_{2}}\big) D_{\alpha_1}^{(\ell_k+1)}\\  
  &= \big(D_{\alpha_k}\dots D_{\alpha_3} D_{\alpha_{2}}\big) 
      D_{\alpha_2}^{\ell_{k-1}}\, s_{1,2}^{\ell_{k-1}} \\  
  &= \big(D_{\alpha_k}\dots D_{\alpha_4} D_{\alpha_{3}}\big) 
      D_{\alpha_2}^{(\ell_{k-1}+1)}\, s_{1,2}^{\ell_{k-1}} \\ 
  &= \big(D_{\alpha_k}\dots D_{\alpha_4} D_{\alpha_{3}}\big) 
      D_{\alpha_3}^{\ell_{k-2}} \, s_{2,3}^{\ell_{k-2}} \, s_{1,2}^{\ell_{k-1}} \\        
  & \qquad \ldots \\
  &= D_{\alpha_k} D_{\alpha_k}^{\ell_1}\, 
    s_{k-1, k}^{\ell_1} \dots s_{2,3}^{\ell_{k-2}} \, s_{1,2}^{\ell_{k-1}} \\   
  &= t_{k+1} \, s_{k,k+1} \, s_{k-1, k}^{\ell_1} \dots s_{2,3}^{\ell_{k-2}} \, s_{1, 2}^{\ell_{k-1}}. 
\end{align*} 
Therefore, 
\begin{align*}
\norm{u_k \phi^{k/5}}_\Sn 
   &=\ell_{k-1} + \dots \ell_1 + 2\\
   & =  n^{k-1} - n^{k-3} - 2n^{k-4} - \ldots - (k-3)n - (k-2) + 1.
\end{align*}
This is because the coefficient of $n^i$ is $1$ is $\ell_i$ and is $(-1)$ in 
$\ell_k, \dots, \ell_{i+1}$. Summing up, we get $-(k-i-1)$ as the coefficient of $n^i$. 
\end{proof}


\begin{figure}[ht]
\setlength{\unitlength}{0.01\linewidth}
\begin{picture}(100,120)
\put(5,95){\includegraphics[width=25\unitlength]{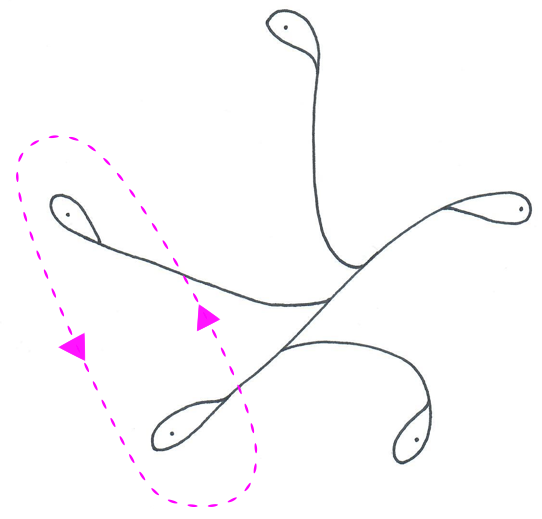}}
\put(38,97){\includegraphics[width=23\unitlength]{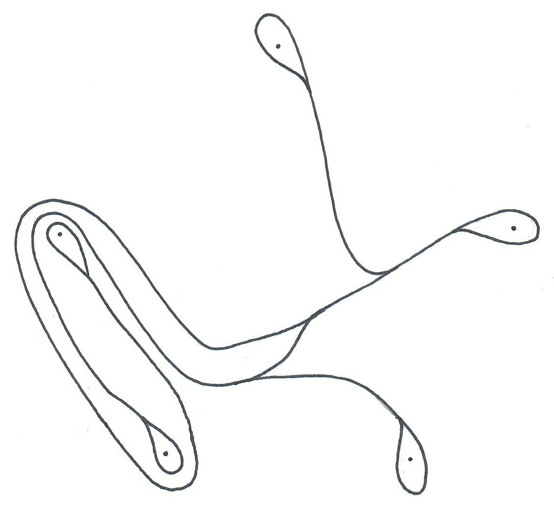}}
\put(70,95){\includegraphics[width=25\unitlength]{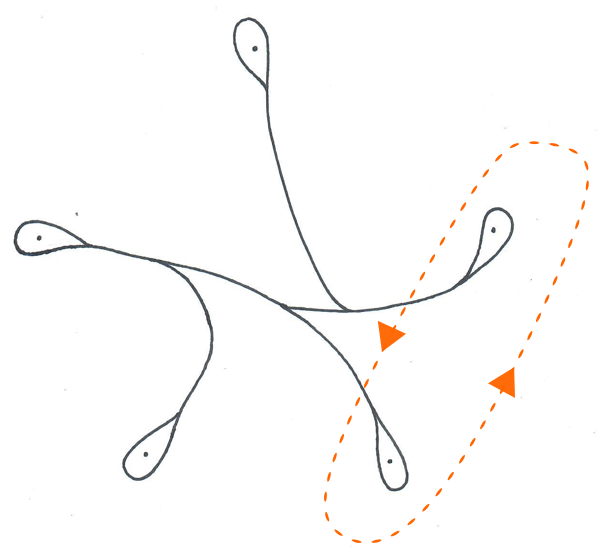}}
\put(72,66){\includegraphics[width=23\unitlength]{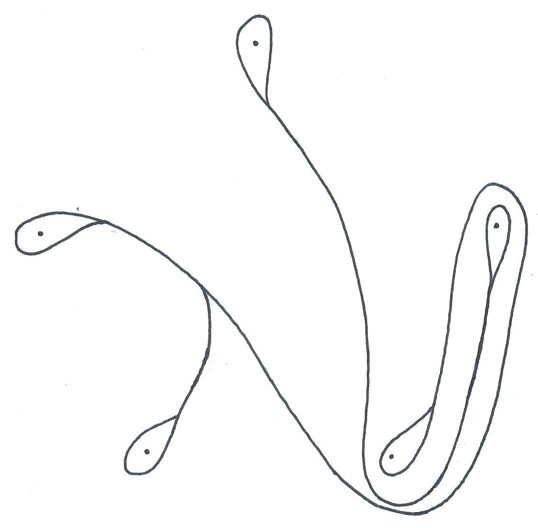}}
\put(38,66){\includegraphics[width=27\unitlength]{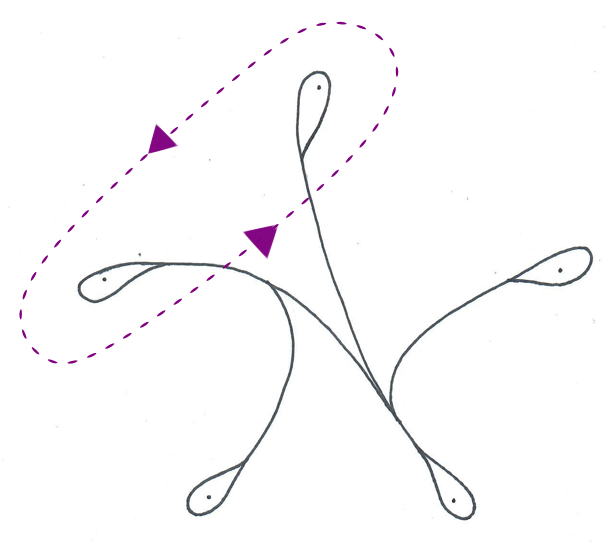}}
\put(7,66){\includegraphics[width=23\unitlength]{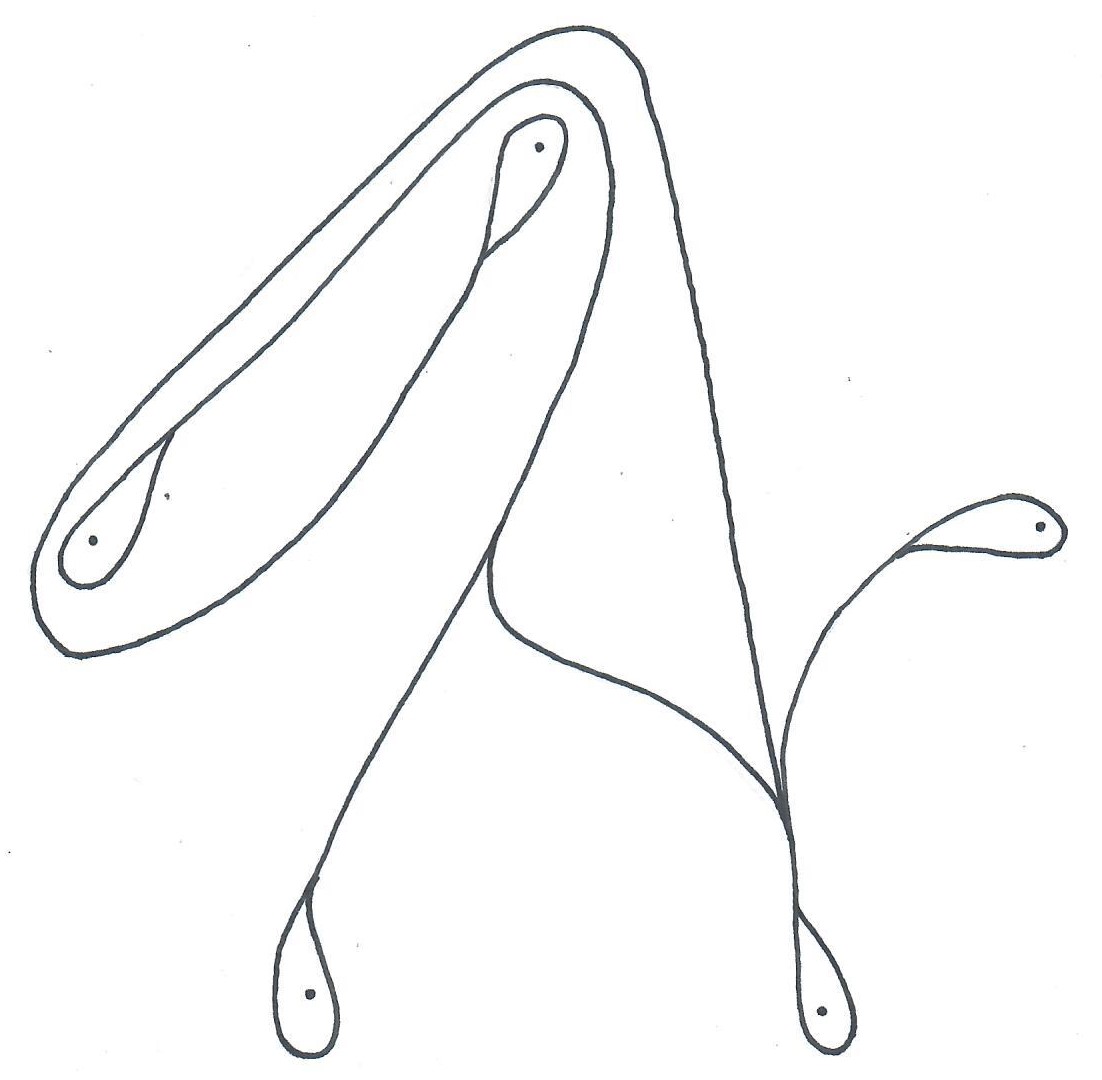}}
\put(5,32){\includegraphics[width=26\unitlength]{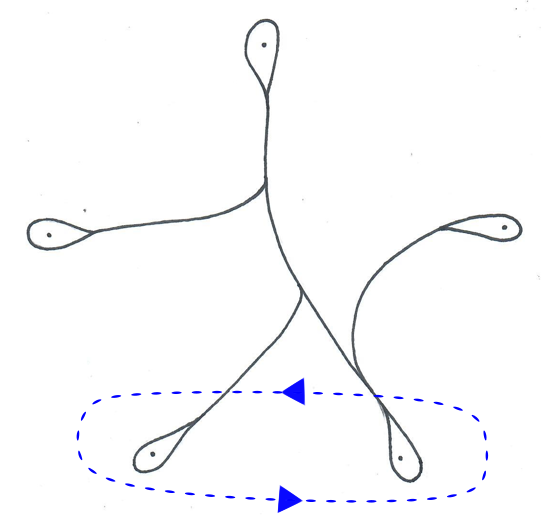}}
\put(38,33){\includegraphics[width=25\unitlength]{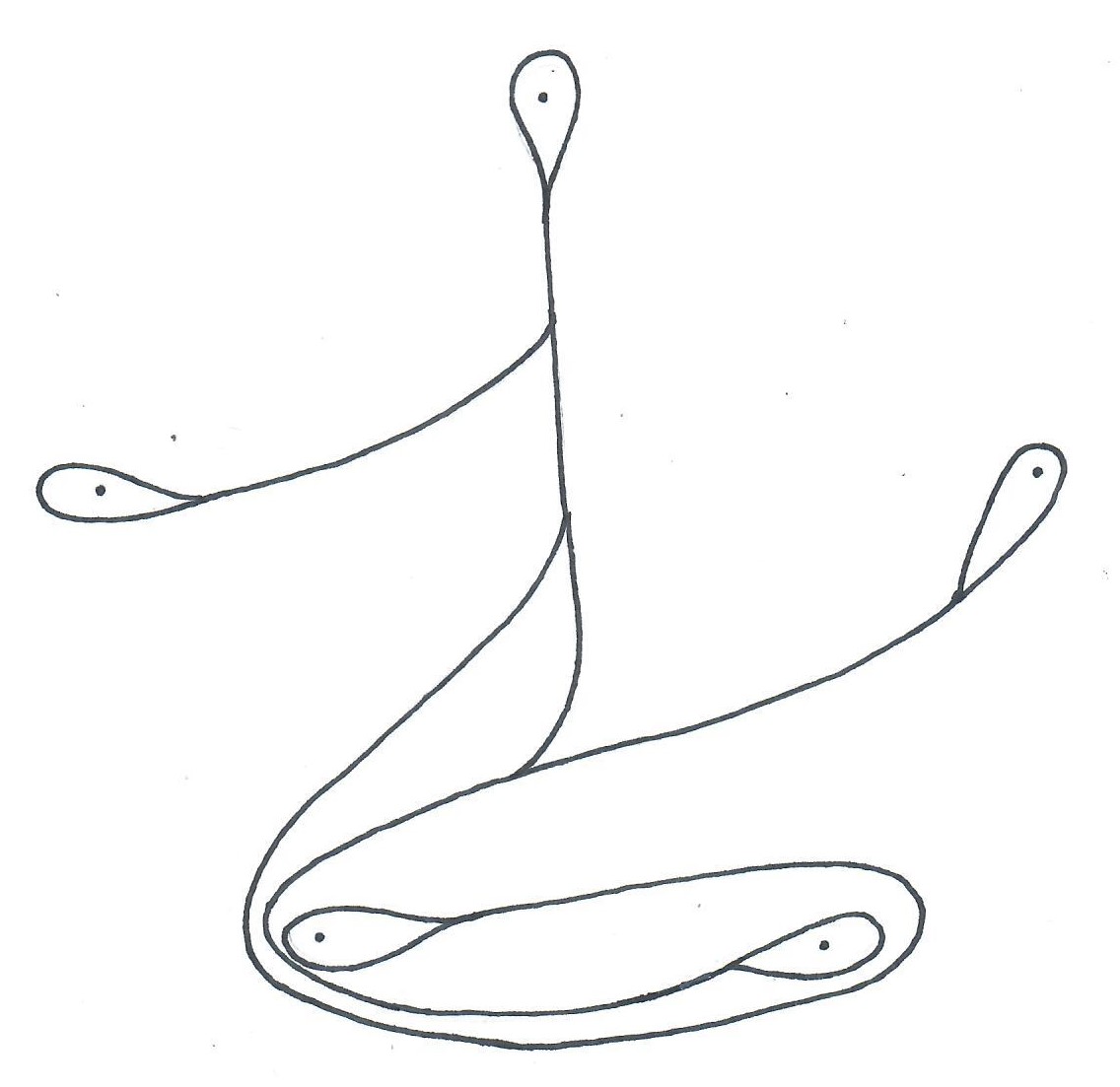}}
\put(70,34){\includegraphics[width=28\unitlength]{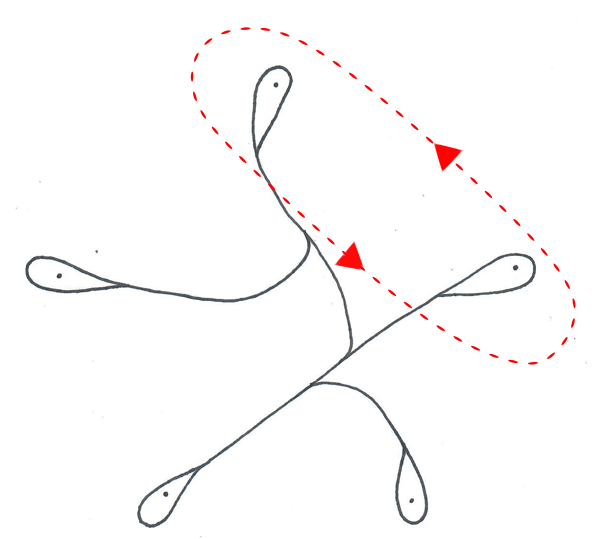}}
\put(70,3){\includegraphics[width=25\unitlength]{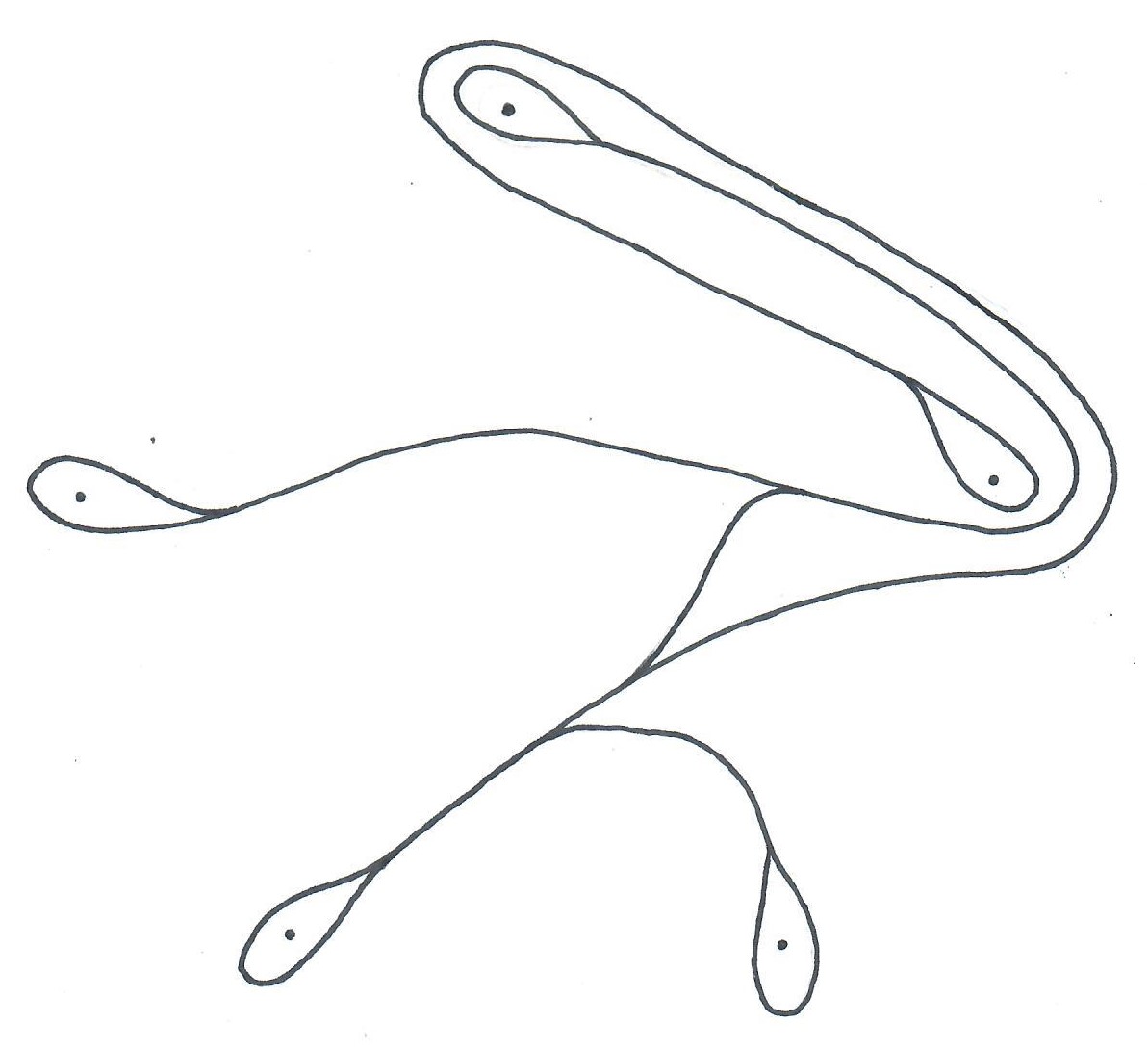}}
\put(38,3){\includegraphics[width=25\unitlength]{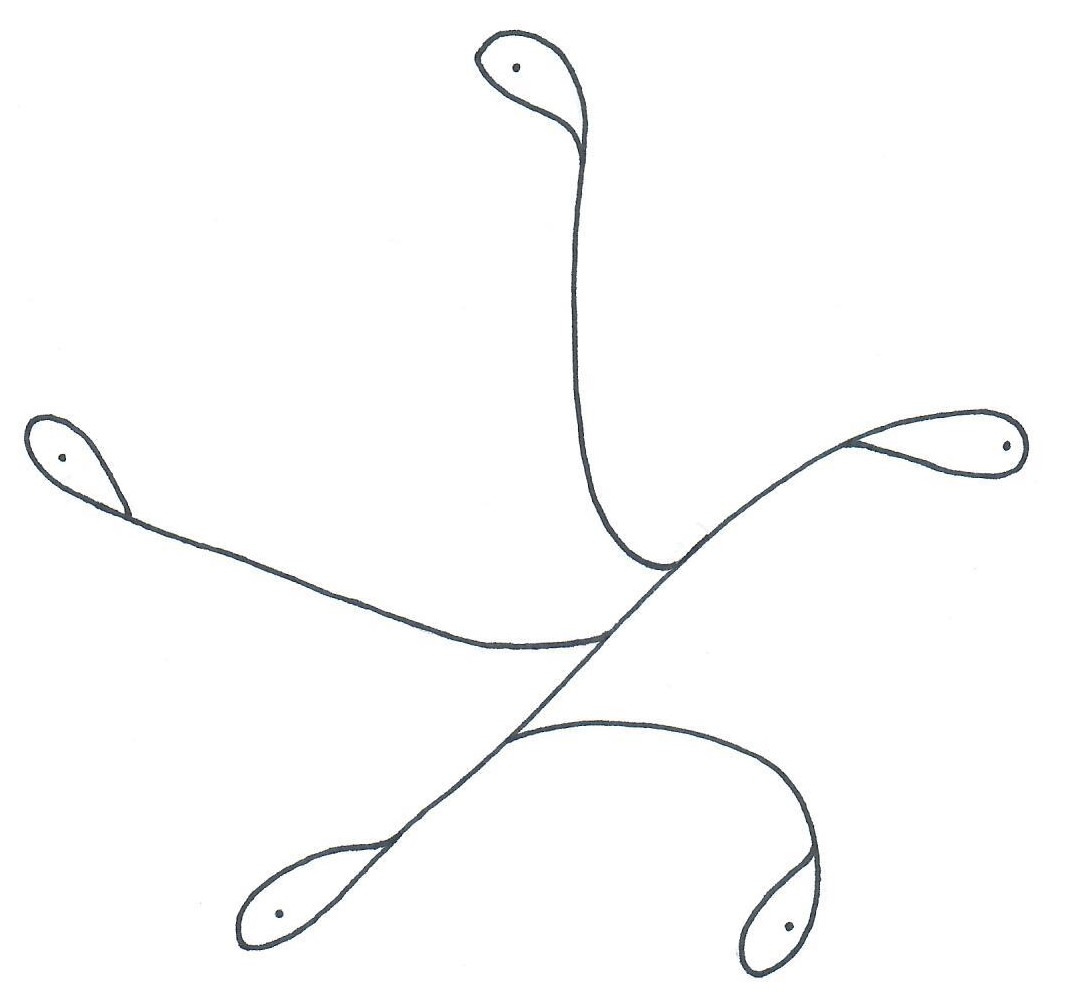}}

\put(7,111){\tiny{$a$}}
\put(14,96){\tiny{$b$}}
\put(25,97){\tiny{$c$}}
\put(29,111){\tiny{$d$}}
\put(17, 119){\tiny{$e$}} 

\put(30,105){$\longrightarrow$}

\put(34,111){\tiny{$3a+2b$}}
\put(39,95){\tiny{$2a+b$}}
\put(56,96){\tiny{$c$}}
\put(59,111){\tiny{$d$}}
\put(47, 119){\tiny{$e$}}

\put(64,105){$\longrightarrow$}

\put(67,111){\tiny{$3a+2b$}}
\put(70,96){\tiny{$2a+b$}}
\put(89,97){\tiny{$c$}}
\put(92,110){\tiny{$d$}}
\put(83, 119){\tiny{$e$}}

\put(81,93){$\big\downarrow$}

\put(70,80){\tiny{$3a+2b$}}
\put(73,64){\tiny{$2a+b$}}
\put(90,64){\tiny{$3c+2d$}}
\put(92,81){\tiny{$2c+d$}}
\put(83, 89){\tiny{$e$}}

\put(66,74){$\longleftarrow$}

\put(38,79){\tiny{$3a+2b$}}
\put(42,64){\tiny{$2a+b$}}
\put(55,64){\tiny{$3c+2d$}}
\put(59,80){\tiny{$2c+d$}}
\put(52, 88){\tiny{$e$}} 

\put(32,74){$\longleftarrow$}

\put(0,73){\tiny{$3a+2b+2e$}}
\put(7,64){\tiny{$2a+b$}}
\put(22,64){\tiny{$3c+2d$}}
\put(26,80){\tiny{$2c+d$}}
\put(11, 89){\tiny{$6a+4b+3e$}} 

\put(17,62){$\big\downarrow$}

\put(0,44){\tiny{$3a+2b+2e$}}
\put(7,33){\tiny{$2a+b$}}
\put(22,32){\tiny{$3c+2d$}}
\put(26,48){\tiny{$2c+d$}}
\put(6, 55){\tiny{$6a+4b+3e$}} 

\put(32,43){$\longrightarrow$}

\put(35,450){\tiny{$3a+2b+2e$}}
\put(35,33){\tiny{$6a+3b+ 6c+4d$}}
\put(52,39){\tiny{$4a+2b+3c+2d$}}
\put(58,50){\tiny{$2c+d$}}
\put(39, 56){\tiny{$6a+4b+3e$}} 

\put(65,43){$\longrightarrow$}

\put(70,50){\tiny{$3a+2b+2e$}}
\put(66,33){\tiny{$6a+3b+ 6c+4d$}}
\put(85,33){\tiny{$4a+2b+3c+2d$}}
\put(95,50){\tiny{$2c+d$}}
\put(84, 58){\tiny{$6a+4b+3e$}} 

\put(82,29){$\big\downarrow$}

\put(70,17){\tiny{$3a+2b+2e$}}
\put(68,2){\tiny{$6a+3b+ 6c$}}
\put(71,0){\tiny{$+4d$}}
\put(84,1){\tiny{$4a+2b+3c+2d$}}
\put(88,11){\tiny{$12a+8b+6c$}}
\put(90,9){\tiny{$+3d+6e$}}
\put(83, 25){\tiny{$6a+4b+4c+2d+3e$}} 

\put(64,12){$\longleftarrow$}

\put(33,17){\tiny{$3a+2b+2e$}}
\put(33,0){\tiny{$6a+3b+ 6c+4d$}}
\put(52,0){\tiny{$4a+2b+3c+2d$}}
\put(54,19){\tiny{$12a+8b+6c$}}
\put(56,17){\tiny{$+3d+6e$}}
\put(32, 23){\tiny{$6a+4b+4c+2d+3e$}} 
\end{picture} 
\caption{The train track $\phi(\tau)$ is carried by $\tau$.}
\label{Fig:pA1} 
\end{figure}
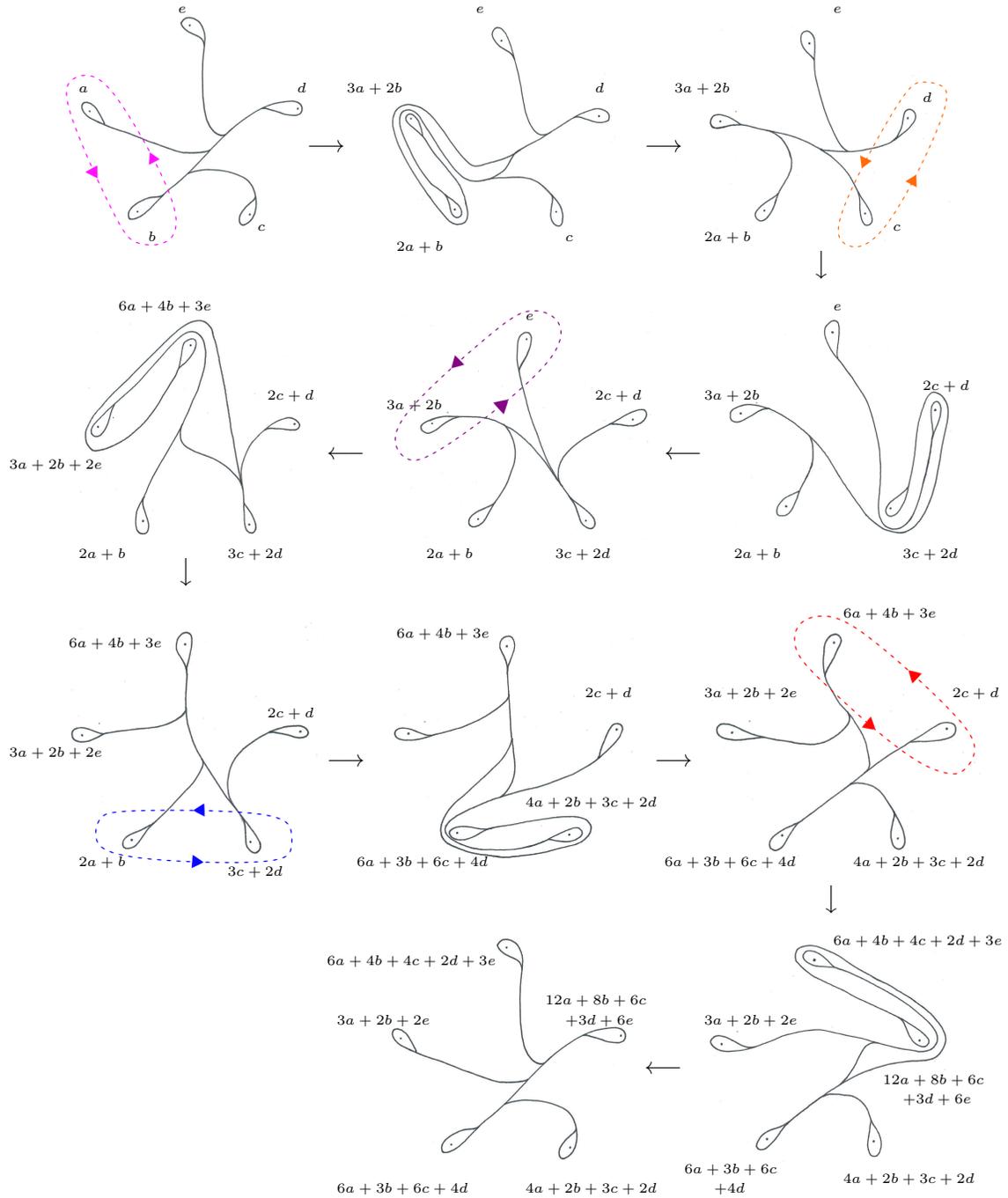

\section{The pseudo-Anosov Map $\phi$} \label{Sec:pA}

In this section, we introduce the pseudo-Anosov map $\phi$ which will be used 
in the proof of Theorem~\ref{Thm:Main-Theorem}. Define
\[
\phi = D_{\alpha_5} D_{\alpha_4} D_{\alpha_3} D_{\alpha_2}D_{\alpha_1}.
\]
We check that $\phi$ is, in fact, a pseudo-Anosov.

\begin{theorem}
The map $\phi$ is pseudo-Anosov.
\end{theorem}

\begin{proof}
In order to prove that $\phi$ is a pseudo-Anosov map, we find a train track $\tau$
on $S$ so that $\phi(\tau)$ is carried by $\tau$ and show that the matrix representation 
of $\phi$ in the coordinates given by $\tau$ is a Perron-Frobenius matrix
(see \cite{PH} for basic information about train-tracks). 

The series of images in Fig.~\ref{Fig:pA1} depict the train track $\tau$ and
its images under successive applications of Dehn twists associated to $\phi$. 
We see that $\phi(\tau)$ is indeed carried by $\tau$ and, keeping track of weights on 
$\tau$, we calculate that the induced action on the space of weights on $\tau$ is
given by the following matrix. 

\[ A= \left( \begin{array}{ccccc}
3 & 2 & 0 & 0 & 2 \\
6 & 3 & 6 & 4 & 0 \\
4 & 2 & 3 & 2 & 0 \\
12 & 8 & 6 & 3 & 6 \\
6 & 4 & 4 & 2 & 3  \end{array} \right)\]

Note that the space of admissible weights on $\tau$ is the subset of $\RR^5$ given
by positive real numbers $a, b, c, d$ and $e$ such that $a+b+e = c+d$. The linear 
map described above preserves this subset. The square of the matrix $A$ is strictly 
positive, which implies that the matrix is a Perron-Frobenius matrix. 
In fact, the top eigenvalue is
\[
\lambda = \sqrt{13}+ 2\, \sqrt{2  \sqrt{13} + 7} + 4
\]
that is associated to a unique irrational measured lamination $F$ carried by $\tau$
that is fixed by $\phi$.  We now argue that $F$ is filling. Note that, curves on $S$ are in 
one-to-one association with simple arcs connecting one puncture to another. We say an 
arc is carried by $\tau$ if the associated curve is carried by $\tau$. If $F$ is not filling, 
it is disjoint from some arc $\omega$ connecting two of the punctures. Modifying 
$\omega$ outside of a small neighborhood of $\tau$, we can produce an arc that is 
carried by $\tau$. In fact, for any two cusps of the train-track $\tau$, either an arc going 
clock-wise or counter-clockwise connecting these two cusps can be pushed into $\tau$. 
Hence, we can replace the portion of $\omega$ that is outside of a small neighborhood 
of $\tau$ with such an arc to obtain an arc $\omega'$ that is still disjoint from $F$
but is also carried by $\tau$. Hence, if $F$ is not filling, it is disjoint from some arc
(and thus some curve) carried by $\tau$. But $F$ is the unique lamination carried by
$\tau$ that is fixed under $\phi$ which is a contradiction. 
This implies that $\phi$ is pseudo-Anosov. 
\end{proof}

\section{Shadow to Curve Complex not a Quasi-Geodesic}\label{Sec:Shadows}
The curve graph $\CS$ is a graph whose vertices are curves on $S$ and
whose edges are pairs of disjoint curves. We assume every edge has length one
turning $\CS$ into a metric space. This means that, for a pair of curves $\alpha$ and $\beta$, $d_\CS(\alpha, \beta) = n$ if 
\[
\alpha = \gamma_0, \ldots, \gamma_n = \beta
\] 
is the shortest sequence of curves on $S$ such that the successive $\gamma_i$ are 
disjoint. Masur-Minsky showed that $\CS$ is an infinite diameter Gromov hyperbolic
space \cite{MM1}. 

We also talk about the distance between subsets of $\CS$ using the
same notation. That is, for two sets of curves $\mu_0, \mu_1 \subset \CS$ we
define 
\begin{equation*}
d_\CS(\mu_0,\mu_1) = \max_{\gamma_0 \in \mu_0, \gamma_1 \in \mu_1} d_\CS(\gamma_0,\gamma_1).
\end{equation*}

\begin{definition} \label{Def:Shadow}
The \emph{shadow map} from the mapping class group to the curve complex is the map
defined as:
\begin{align*}
\Upsilon \from \Map(S) &\to \CS\\
f &\to f(\alpha_1).
\end{align*}
\end{definition}

The shadow map from $\Map(S)$ equipped with $d_\Sn$ to the curve complex
is $4$-Lipschitz:

\begin{lemma}\label{lemma: cc to map}
For any $f \in \Map(S)$, we have
\begin{equation}
d_\CS (\alpha_1, f \alpha_1) \leq 4 \norm{f}_\Sn. 
\end{equation}
In particular, the Lipschitz constant of the shadow map is independent of $n$. 
\end{lemma}

\begin{proof}
It is sufficient to prove the lemma for elements of $\Sn$. Consider $D_{\alpha_i} \in \Sn$. 
If $i(\alpha_i, \alpha_1)=0$ then 
\[
d_\CS(\alpha_1 , D_{\alpha_i} ( \alpha_1) ) = d_\CS(\alpha_1, \alpha_1)= 0.
\] 
If $i(\alpha_i, \alpha_1)=2$, then there is a curve $\alpha_j$ that disjoint from both 
$\alpha_1$ and $\alpha_i$ and hence $\alpha_j$ is also disjoint $D_{\alpha_i} (\alpha_1)$. Therefore, 
$d_\CS(\alpha_1, D_{\alpha_i} (\alpha_1))=2$. 

Now consider the element $s_{i,i+1} \in \Sn$. Note that 
$s_{i, i+1}^{-1}\alpha_i = \alpha_i$. Hence, 
\begin{align*}
d_\CS(\alpha_1,s_{i,i+1} \alpha_1) 
  &\leq d_\CS(\alpha_1, \alpha_i) + d_\CS(\alpha_i, s_{i,i+1} \alpha_1)\\
  &\leq 2 + d_\CS(s_{i, i+1}^{-1} \alpha_i , \alpha_1)\\     
&\leq 2 + d_\CS(\alpha_i, \alpha_1) \leq 2+2=4
\end{align*}
Thus, we have proven our claim.
\end{proof}

Using this lemma and the theorems from Section 3, we show that the shadow of geodesics 
from the mapping class group to the curve complex are not always quasi-geodesics.

\begin{theorem}
For all $K \geq 1$, $C \geq 0$, there exists a geodesic in the mapping class group, 
whose shadow to the curve complex is not a $(K,C)$--quasi-geodesic.
\end{theorem}

\begin{proof}
Recall that, for a positive integer $k$, we have 
\[
m_k = n^{k-1} + n^{k-2} + \ldots + n + 1, 
\qquad
\ell_k = n^k - n^{k-1} - n^{k-2} - \ldots - n -1,
\]
$w_k = D_{\alpha_1}^{m_k}$ and $u_k = D_{\alpha_1}^{\ell_k}$. Note
that $m_{k-1} + \ell_k = n^k$. Hence, we can write 
\[
D_{\alpha_1}^{n^k} = \Big(w_{k-1} \phi^{-k/5}\Big) \Big( \phi^{k/5} u_{k} \Big).
\]
Also,
\[
h\Big(D_{\alpha_1}^{n^k}\Big) = n^k=(n-1)\big(n^{k-1} + n^{k-2} + \ldots + n + 1\big) +1.
\]
Therefore by Lemma \ref{cor1}
\begin{equation}\label{shadow eqn1}
\norm{D_{\alpha_1}^{n^k}}_\Sn \geq n^{k-1} + n^{k-2} + \ldots + n + 2.
\end{equation}
But, from Theorem \ref{thm:geodesics} we have
\begin{equation*}
\norm{w_{k-1} \phi^{-k/5}}_\Sn = n^{k-2} + 2n^{k-3} + \ldots + (k-2)n + (k-1).
\end{equation*}
and 
\begin{equation*}
\norm{u_k \phi^{k/5}}_\Sn = n^{k-1} - n^{k-3} - 2n^{k-4} - \ldots - (k-3)n - (k-2) + 1.
\end{equation*}
The sum of the word lengths of the two elements is
\begin{equation*}
n^{k-1} + n^{k-2} + \ldots + n + 2
\end{equation*}
which is equal to the lower bound found in Equation \ref{shadow eqn1}. 
Thus
\[
\norm{D_{\alpha_1}^{n^k}}_\Sn = 
\norm{w_{k-1} \phi^{-k/5}}_\Sn + \norm{\phi^{k/5} u_{k}}_\Sn
\]
which means there is a geodesic connecting $D_{\alpha_1}^{n^k}$ to the identity
that passes through $\phi^{k/5}u_k$. 

Since $\phi$ is a pseudo-Anosov map, there is a lower-bound on its translation 
distance along the curve graph (see Theorem 4.6 from \cite{MM1}). Namely, there 
is a constant $\sigma>0$ so that, for every $m$, 
\begin{equation} \label{Eq:MTD}
d_\CS\big(\alpha_1,\phi^m \alpha_1\big) \geq \sigma \, m.
\end{equation}
Also, $u_k \alpha_1 = \alpha_1$ which implies 
\begin{equation*}
d_\CS\big(\alpha_1, \phi^{k/5}u_k \alpha_1\big)
= d_\CS\big(\alpha_1, \phi^{k/5} \alpha_1\big) \geq \sigma \, \frac{k}{5}.
\end{equation*}
That is,
\[
\Upsilon (\id) = \Upsilon(D_{\alpha_1}^{n^k})= \alpha_1.
\]
However, $\Upsilon\big(\phi^{k/5}u_k\big)$ is at least distance $\frac{\sigma k}5$ away from 
$\alpha_1$. Therefore, choosing $k$ large compared with $\sigma$, $K$ and $C$, we see that 
the shadow of this geodesic (the one connecting $\id$ to $D_{\alpha_1}^{n^k}$ which passes
through $\phi^{k/5}u_k$) to $\CS$ is not a $(K,C)$--quasi-geodesic. 
\end{proof}


\section{Axis of a pseudo-Anosov in the mapping class group}\label{Sec:pA_axis}

Consider the path 
\[
\Aphi \from \ZZ \to \Map(S), \qquad i \to \phi^{i}.
\]
Since $\norm{\phi}_\Sn \leq 5$, then $\norm{\phi^i}_\Sn \leq 5i$. 
Also, using Lemma \ref{lemma: cc to map} and Equation~\eqref{Eq:MTD} we get
\[
\norm{\phi^i}_\Sn 
\geq \frac 14 d_{C(S)}\big(\alpha_1, \phi^i \alpha_1 \big)
\geq  \frac{i \, \sigma}4. 
\]
Therefore,
\[
\frac{i \sigma}{4} \leq \norm{\phi^i}_\Sn \leq 5i.
\]
This proves the following lemma.
\begin{lemma}\label{Lem:quasi-geodesic} 
The path $\Aphi$ is a quasi-geodesic in $(\Map(S), d_\Sn)$ for every $n$ with
uniform constants. 
\end{lemma}

We abuse notation and allow $\Aphi$ to denote both the map, and the image of the map 
in $\Map(S)$. For $i,j\in \ZZ$, let $\calg = \calg_{i,j}$
be a geodesic in  $(\Map(S), d_\Sn)$ connecting $\phi^i$ to $\phi^{\, j}$. 
Let $\calG= \Upsilon \circ \calg$ be the shadow of $\calg$ to the curve complex and let 
\[\Proj_\calG \from \Map(S) \to \calG\]
be the composition of $\Upsilon$ and the closest point projection from $\CS$ to 
$\calG$. The following theorem, proven in more generality by Duchin and Rafi 
\cite[Theorem 4.2]{DR1}, is stated for geodesics $\calg_{i,j}$ and the path $\calG$.

\begin{theorem}\label{Thm:DR}
The path $\calG$ is a quasi-geodesic in $\CS$. Furthermore, 
there exists a constant $B_n$ which depends on $n$ and $\phi$, and a constant $B$ 
depending only on $\phi$ such that the following holds. For $x \in \Map(S)$ with 
$d_\Sn(x, \calg) > B_n$, let $r=d_\Sn(x, \calg) / B_n$
and let $B(x,r)$ be the ball of radius $r$ centered at $x$ in $(\Map(S), d_\Sn)$. Then 
\begin{equation*}
\diam_\CS\!\Big(\Proj_\calG\! \big(B(x,r)\big)\Big) \leq B.
\end{equation*}
\end{theorem}

In the proof of \cite[Theorem 4.2]{DR1}, it can be seen that $B_n$ 
($B_1$ in their notation) is dependent on the generating set since $B_n$ is taken to 
be large with respect to the constants from the Masur and Minsky distance formula which 
depend on the generating set \cite{MM2}. Let $\calS$ be a fixed generating set
for $\Map(S)$. Then the word lengths of elements in $\Sn$ in terms of $\calS$ 
grow linearly in $n$ with respect to $\calS$. Hence, the constants involved in the 
Masur-Minsky distance formula also change linearly in $n$. That is, $B_n \asymp n$. 
Also, one can see that the constant $B$ ($B_2$ in their proof) depends only on 
$\phi$ and the hyperbolicity constant of the curve graph, but not the generating set.

Since, $\Aphi$ is a quasi-geodesic, \thmref{Thm:DR} and the usual Morse argument 
implies the following.

\begin{proposition} \label{Prop:delta_n}
The paths $\Aphi[i,j]$ and $\calg_{i,j}$ fellow travel each other and the constant
depends only on $n$. That is, there is a bounded constant $\delta_n$ depending on 
$n$ such that 
\[
\delta_n \geq \max \left(
 \max_{\substack{p \in \Aphi[i,j]}}  \min_{\substack{q \in \calg_{i,j}}}  d_\Sn (p,q), 
 \max_{\substack{p \in \calg_{i,j}}}   \min_{\substack{q \in \Aphi[i,j]}}  d_\Sn (p,q) \right).
\]
\end{proposition}

We now show that $\phi$ acts loxodromically in $(\Map(S), d_\Sn)$. That 
is, there exists a geodesic $\aphi$ in  $(\Map(S), d_\Sn)$ that is preserved by a
power of $\phi$. This is folklore theorem, but we were unable to find a reference for it 
in the literature. The proof given here follows the arguments in  \cite[Theorem 1.4]{Bo} 
where Bowditch showed that $\phi$ acts loxodromically on the curve graph,
which is more difficult since the curve graph is not locally finite. Bowditch's 
proof in turn follows the arguments of Delzant \cite{D} for a hyperbolic group.

\begin{proposition} \label{Prop:qa} 
There is a geodesic 
\[\aphi \from \ZZ \to \Map(S)\]
that is preserved by some power of $\phi$. We call the geodesic $\aphi$ the quasi-axis for 
$\phi$. 
\end{proposition}

\begin{proof}
The statement is true for the action of any pseudo-Anosov homeomorphism in
any mapping class group equipped with any word metric coming from a finite generating 
set. We only sketch the proof since it is a simpler version of the argument given in \cite{Bo}. 

Let $\calL(i,j)$ be the set of all geodesics connecting $\phi^i$ to $\phi^{\, j}$. Note
that every point on every path in $\calL(i,j)$ lies in the $\delta_n$--neighborhood of
$\Aphi$.  Letting $i \to \infty$, $j \to -\infty$ and using a diagonal limit argument
($\Map(S)$ is locally finite) we can find bi-infinite geodesics that are the limits of geodesic 
segments in sets $\Aphi[i,j]$. Let $\calL$ be the set of all such  bi-infinite geodesics. 
Then $\phi(\calL)=\calL$ and every geodesic in $\calL$ is also contained in the
$\delta_n$--neighborhood of $\Aphi$. Let $\calL/\phi$ represent the set of edges
which appear in a geodesic in $\calL$ up to the action of $\phi$. Then $\calL/\phi$
is a finite set. 

Choose an order for $\calL/\phi$. We say a geodesic $\calg \in \calL$
is lexicographically least if for all vertices $x, y \in \calg$, the sequence of 
$\phi$-classes of directed edges in the segment $\calg_0 \subset \calg$ between 
$x$ and $y$ is lexicographically least among all geodesic segments from $x$ to $y$ 
that are part of a geodesic in $\calL$.  Let $\calL_L$ be the set lexicographically least 
elements of $\calL$. We will show that every element of $\calL_L$ is preserved
by a power of $\phi$. 

Let $P$ be the cardinality of a ball of radius $\delta_n$ in $(\Map(S), d_\Sn)$. We claim 
that $|\calL_L| \leq P^2+1$. Otherwise, we can find $P^2 + 1$ elements of $\calL_L$
which all differ in some sufficiently large compact subset $N_{\delta_n}(\Aphi)$,
the $\delta_n$--neighborhood of $\Aphi$.  In particular, we can find 
$x, y \in N_{\delta_n}(\Aphi)$ so that each of these $P^2 + 1$ geodesics has a 
subsegment connecting a point in $N_{\delta_n}(x)$ to a point in $N_{\delta_n}(y)$, 
and these subsegments are all distinct. But then, at least two such segments must 
share the same endpoints, which means they cannot both be lexicographically least. 

Since $\phi$ permutes elements of $\calL_L$,  each geodesic in $\calL_L$
is preserved by $\phi^{P^2+1}$. 
\end{proof}

As before, we use the notation 
$\aphi$ to denote both the map and the image of the map in $\Map(S)$. 
We now show that the projection of a ball that is disjoint from 
$\aphi$ to $\aphi$ grows at most logarithmically with the radius of the ball
proving that \thmref{Thm:Main-Theorem} is sharp. 

\begin{corollary}\label{Cor:Diam-Proj}
There are uniform constants $c_1, c_2>0$ so that, for $x \in \Map(S)$ 
and $R= d_\Sn(x, \aphi)$, we have 
\begin{equation*}
\diam_\CS\!\Big(\PG \!\big(\Ball(x, R) \big) \Big) 
\leq c_1 \log (n) \log(R) + c_2 n.
\end{equation*}
\end{corollary}

\begin{proof}
Consider $y \in \Ball(x, R-B_n)$. Let $N$ be the smallest number so that there is 
a sequence of points along the geodesic connecting $x$ to $y$ 
\[
x=x_0, x_1, \dots, x_N=y
\]
so that 
\[
d_\Sn(x_i,x_{i+1}) \leq \frac{d_\Sn(x_i,\aphi)}{B_n}.
\] 
Then, 
\begin{align*}
d_\Sn(x_{i+1}, \aphi) 
&\geq d_\Sn(x_i, \aphi) - d_\Sn(x_i,x_{i+1}) \\
&\geq d_\Sn(x_i, \aphi) -\frac{d_\Sn(x_i,\aphi)}{B_n}
\geq d_\Sn(x_i, \aphi) \left(1 - \frac 1{B_n}\right). 
\end{align*}
Hence, 
\[
d_\Sn(x_i, \aphi)  \geq R \left(1 - \frac 1{B_n}\right)^i. 
\]
Since $N$ is minimum
\[
d_\Sn(x_i,x_{i+1}) +1 \geq \frac{d_\Sn(x_i,\aphi)}{B_n},
\]
which implies
\[
d_\Sn(x_i,x_{i+1}) \geq  \frac{R}{B_n} \left(1 - \frac 1{B_n}\right)^i -1.
\]

Since $d_\Sn(x, y) \leq R - B_n$, 
\[
d_\Sn(x_i, \aphi) \geq R - d_\Sn(x, x_i) \geq R - d_\Sn(x, y)  \geq B_n. 
\]
Applying Theorem \ref{Thm:DR} to $r_i = d_\Sn(x_i, \aphi)/B_n$ and 
$x_{i+1} \in \Ball(x_i, r_i)$ we get 
\[
d_\CS\big(\PG(x_i), \PG(x_{i+1}) \big) \leq B,
\]
and hence, 
\begin{equation} \label{Eq:y}
d_\CS\big(\PG (x), \PG(y) \big)  \leq B c_n' \log R. 
\end{equation}
Now, for any $y' \in \Ball(x, R)$ there is a $y \in \Ball(x, R-B_n)$ with 
$d_\Sn(y, y') \leq B_n$. But $\Upsilon$ is $4$--Lipschitz and the closest
point projection from $\CS$ to $\Gphi$ is also Lipschitz with a Lipschitz
constant depending on the hyperbolicity constant of $\CS$. Therefore, 
\begin{equation}  \label{Eq:y'}
d_\CS\big(\PG(y), \PG(y')\big) \leq c'' B_n, 
\end{equation} 
where $c''$, the Lipschitz constant for $\PG$, is a uniform constant. 
By letting 
\[c_n = \max(B c_n' , B_n c'') \asymp \log(n),\]
the Corollary follows from \eqnref{Eq:y} and \eqnref{Eq:y'} and the triangle inequality. 
\end{proof} 


\section{The logarithmic lower-bound} \label{Sec:Upper_bound}
In this section, we will show that the quasi-axis $\aphi$ of the 
pseudo-Anosov map $\phi$ does not have the strongly contracting property
proving \thmref{Thm:Main-Theorem} from the introduction.

\begin{definition} \label{Def:Projection}
Given a metric space $(X, d_X)$, a subset $\calG$ of $X$ and constants
$d_1, d_2 >0$, we say a map $\Proj\from X \to \calG$, 
a $(d_1,d_2)$--projection map if for every $x \in X$ and $g \in \calG$, 
\[
d_X\big(\Proj (x), g \big) \leq d_1 \cdot d_X(x,g) + d_2. 
\]
\end{definition}

To prove this theorem, notice first that the geodesic found in Section 2 may not determine 
the nearest point of $\Aphi$ to $w_k = D_{\alpha_1}^{m_k}$, where $m_k = n^{k} + n^{k-1} + \ldots + n + 1$.

\begin{lemma}
If $\phi^{p_k}$ is the nearest point of $\Aphi$ to $w_k$, then $p_k \geq k/5$.
\end{lemma}

\begin{proof}
Consider a point $\phi^m$ on $\Aphi$ where $m < k/5$. 
Applying the homomorphism $h$ we have
\[
h(w_k \phi^{-m}) = (m_k - 5\, m) > (m_k - k) = h(w_k \phi^{-k/5}).
\]
But $(m_k - k)$ is divisible by $(n-1)$. Hence, if we write $(m_k - m)= q (n-1) + r$,
where $|r| \leq \frac{(n-1)}2$, we have 
\[
|q| \geq \frac{m_k - k}{n-1}, \qquad\text{and}\qquad |r| \geq 0. 
\]
Lemma \ref{cor1} implies that 
$\norm{w_k \phi^{-m}}_\Sn > \norm{w_k \phi^{-k/5}}_\Sn$, which means the closest point 
in $\calA_\phi$ to $w_k$ is some point $\phi^{p_k}$ where $p_k \geq k/5$. 
\end{proof}

\noindent Let $R_k = d_{\Sn}(w_k, \phi^{p_k}) = d_{\Sn}(w_k, \calA_\Phi)$ and $\Delta_k = d_{\Sn}(w_k, \phi^{(k+1)/5})$.

\begin{proof}[Proof of Theorem \ref{Thm:Main-Theorem}]
For fixed $d_1, d_2 > 0$, let $\Proj_\aphi\from \Map(S) \to \aphi$ be any 
$(d_1,d_2)$-projection map. Fix $n$ large enough so that 
\begin{equation} \label{Eq:n} 
\sigma > \frac{5 \, d_1}{n-1}.
\end{equation} 
Choose the sequence $\left\{ k_{i} \right\}= \left\{ 2n^{i}-3 \right\}$ and recall that 
\[
v_{k_i} = D_{\alpha_1}^{\frac{{k_{i} }+1}{2}} D_{\alpha_2}^{\frac{{k_{i}}+1}{2}}.
\] 
By Example \ref{example1} (notice that $\frac{k_i+1}2 = n^i-1$)
\begin{equation} \label{Eq:v}
d_{\Sn} \left( v_{k_i}, \id \right) = \norm{v_{k_i}}_\Sn=\frac{k_{i}  + 1}{n-1}.
\end{equation}
and by \propref{thm:geodesics}, we have
\[
d_\Sn(w_{k_i}, v_{k_i}) = \Delta_{k_i}. 
\]
Consider a ball $B(w_{k_{i}}, r_i)$ of radius $r_i=R_{k_{i}} - (\delta_n+1)$ around 
$w_{k_i}$. This ball is disjoint from $\aphi$ since $\calA_\phi$ and $\aphi$ 
are $\delta_n$-fellow-travellers by \ref{Prop:delta_n}, and 
$R_{k_{i}} = d_{\Sn}(w_{k_{i}}, \calA_\Phi)$. 
For the rest of the proof, we refer to Figure \ref{Fig:main_theorem}.

\begin{figure}[ht]
\centering
\includegraphics[width=0.5\textwidth]{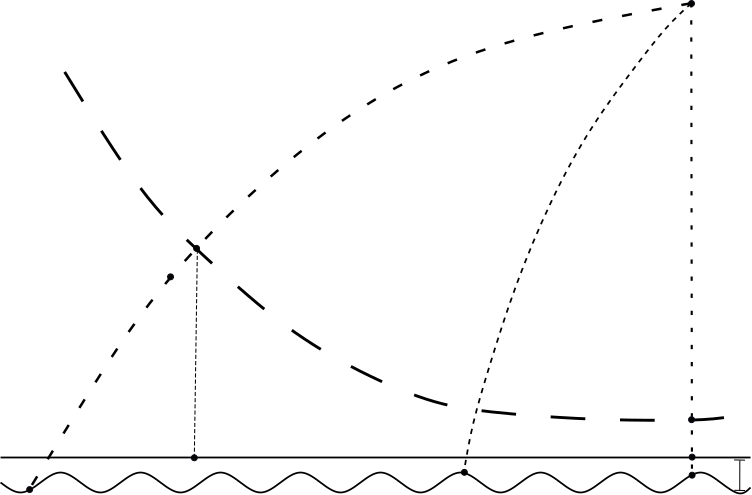}
\put(-225,10){\tiny{$a_{\Phi}$}}
\put(-225,0){\tiny{$A_{\Phi}$}}
\put(-15,60){\tiny{$R_{k_i}$}}
\put(-65,60){\tiny{$\Delta_{k_i}$}}
\put(-15,140){\tiny{$w_{k_i}$}}
\put(-177,65){\tiny{$v_{k_i}$}}
\put(-153,69){\tiny{$p$}}
\put(-15,25){\tiny{$q$}}
\put(-205,-3){\tiny{$\id$}}
\put(-156,16){\tiny{$\Proj_{\aphi}(p)$}}
\put(-51.5,16){\tiny{$\Proj_{\aphi}(q)$}}
\put(-92,-5){\tiny{$\Phi^{\frac{k_i}{5}}$}}
\put(-21,-5){\tiny{$\Phi^{P_{k_i}}$}}
\put(0,3){\tiny{$\delta_n$}}
\put(-222,100){\tiny{$B(w_{k_{i}}, r_i)$}}
\caption{Setup for the proof of Theorem \ref{Thm:Main-Theorem}}
\label{Fig:main_theorem} 
\end{figure}
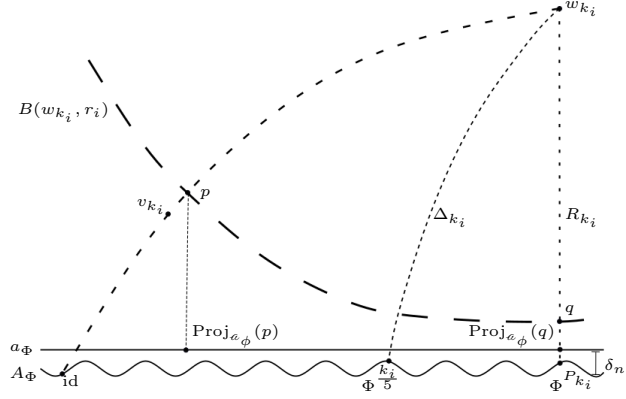

Since $h$ is a homomorphism, we have 
\[
h(w_{k_{i}} \phi^{{k_{i}}/5}) = h(w_{k_{i}} \phi^{-p_{k_{i}}}) + h(\phi^{p_{k_{i}}} \phi^{-{k_{i}}/5})
\]
Theorem \ref{thm:geodesics} showed
\[
h(w_{k_{i}}  \phi^{k_{i} /5}) = (n-1) \Delta_{k_{i}}, 
\]
from \lemref{cor1}, we have 
\[
h(w_k \phi^{-p_{k_{i}}}) \leq (n-1) R_{k_{i}}
\]
and since $\norm{\phi}_\Sn =5$, we have 
\[
h(\phi^{p_{k_{i}}} \phi^{-{k_{i}}/5}) \leq 5p_{k_{i}} -{k_{i}}.
\]
The above equations imply:
\[
\Delta_{k_{i}} - R_{k_{i}} \leq \frac{5p_{k_{i}} -{k_{i}}}{n-1}.
\]
Consider a point $p$ on the geodesic from $w_{k_{i}}$ to $v_{k_i}$ 
such that $d_{\Sn}(w_{k_{i}}, p) = r_i$, ie. such that  
\[
d_{\Sn} \left( p, v_{k_i} \right) =
\Delta_{k_i}- r_i = \Delta_{k_i}-(R_{k_i}- \delta_n -1) \leq 
\frac{5p_{k_{i}} - {k_{i}}}{n-1}+\delta_n + 1.
\]
This and \eqnref{Eq:v} imply 
\begin{equation*}
\begin{aligned}
d_{\Sn}(\id, p) &\leq  \frac{{k_{i}}+1}{n-1} + \frac{5p_{k_{i}} - {k_{i}}}{n-1} + \delta_n +1\\
&= \frac{5p_{k_{i}} + 1}{n-1} + \delta_n + 1.
\end{aligned}
\end{equation*}
Since $\aphi$ and $\calA_\Phi$ are $\delta_n$-fellow-travellers by \ref{Prop:delta_n}, 
there exists a point $x_0 \in \aphi$ in the $\delta_n$ neighborhood of the identity. 
Thus $d_{\Sn}(p,x_0) \leq \frac{5p_{k_{i}} + 1}{n-1} + 2\delta_n + 1$ and 
\begin{equation}\label{eq61}
\begin{aligned}
d_{\Sn}(\id ,\Proj_{\aphi}(p)) &\leq d_{\Sn}(\id,x_0)+d_{\Sn}(x_0,\Proj_{\aphi}(p))\\
&\leq d_{\Sn}(p,x_0) + d_1 \cdot d_{\Sn}(x_0,p) + d_2\\
&\leq  \frac{5 \, d_1 \, p_{k_{i}} }{n-1} + A_p.\\
\end{aligned}
\end{equation}
where $A_p$ is a constant depending on $\delta_n$, $d_1$ and $d_2$ but is independent
of $k_i$. Similarly, we consider a point $q$ on the geodesic from $w_{k_{i}}$ to 
$\phi^{p_{k_i}}$ such that $d_{\Sn}(w_{k_{i}}, q) = r_i$. Again, since $\aphi$ and 
$\calA_\Phi$ are 
$\delta_n$-fellow-travellers by \ref{Prop:delta_n}, there exists an $x_1 \in \aphi$ 
such that $d_{\Sn}(\phi^{p_{k_i}},x_1) \leq \delta_n$, and thus 
$d_{\Sn}(q,x_1) \leq 2 \delta_n + 1$. Therefore
\begin{equation}\label{eq62}
\begin{aligned}
d_{\Sn}(\phi^{p_{k_i}},\Proj_{\aphi}(q)) &\leq d_{\Sn}(\phi^{p_{k_i}},x_1)+d_{\Sn}(x_1,\Proj_{\aphi}(q))\\
&\leq \delta_n + d_1 \cdot (2\delta_n + 1) + d_2 \leq A_q
\end{aligned}
\end{equation}
where, again, $A_q$ depends on $\delta_n$, $d_1$ and $d_2$ but is 
independent of $k_i$. Since $p,q \in B(w_{k_i}, r_i)$, we have 
\begin{align*}
\diam_{\Sn} \Big(\Proj_{\aphi} \big(B(w_{k_{i}}, r_i)\big)\Big) 
 & \geq d_\Sn \big(\Proj_{\aphi}(p), \Proj_{\aphi}(q) \big) \\
 & \geq d_\Sn \big(\id, \phi^{p_{k_{i}}} \big) - d_\Sn \big(\id, \Proj_{\aphi}(p)\big) 
   - d_\Sn \big(\Proj_{\aphi}(q), \phi^{p_{k_{i}}} \big)
\end{align*}
But $d_{\Sn}(\id,\phi^{p_{k_{i}}}) \geq \sigma p_{k_{i}}$. 
By combining this fact and equations \ref{eq61} and \ref{eq62} we find 
\begin{equation}
\begin{aligned}
\diam_{\Sn} \Big(\Proj_{\aphi} \big(B(w_{k_{i}}, r_i)\big)\Big) 
&\geq \sigma p_{k_{i}}  - \frac{5 \, d_1 \, p_{k_{i}} }{n-1}  - A_p - A_q\\
& = p_{k_i} \left( \sigma - \frac{5 d_1}{n-1} \right)  - A_p - A_q. 
\end{aligned}
\end{equation}
By our assumption on $n$ (\eqnref{Eq:n}) this expression is positive and 
goes to infinity at $p_{k_i} \to \infty$. But, for $n$ large enough, 
$r_i \leq R_{k_i} \leq \Delta_{k_i} \leq n^{k_i}$.  Also, 
$p_{k_{i}} \geq \frac{{k_{i}}}{5}$. Hence, 
\[
 \frac {5 \, p_{k_i}}{\log n} \geq \log(r_i). 
\]
 Hence, there is a constant $c_n$ so that 
\[
\diam_{\Sn} \Big(\Proj_{\aphi} \big(B(w_{k_{i}}, r_i)\big)\Big)  \geq c_n \log r_i. 
\]
This finishes the proof. 
\end{proof}



\end{document}